\documentclass{article}
\usepackage{amsmath, amsfonts, newlfont, latexsym, amscd, amsxtra, amssymb, amsbsy, amsgen, amsopn, fontenc, fontsmpl, amsthm, xypic, amsthm}
\usepackage{pgf,tikz,pgfplots}
\usetikzlibrary{arrows}
\usepackage{color, graphicx}
\newtheorem{definition}{Definition}[section]
\newtheorem{lemma}[definition]{Lemma}
\newtheorem{remark}[definition]{Remark}

\newtheorem{example}[definition]{Example}
\newtheorem{proposition}[definition]{Proposition}
\newtheorem{corollary}[definition]{Corollary}
\newtheorem{theorem}[definition]{Theorem}

\def\reg{\operatorname{reg}}

\begin{document}

% Redefine "plain" pagestyle
\makeatletter      % `@' is now a normal "letter' for LaTeX
\renewcommand{\ps@plain}{%
    \renewcommand{\@oddhead}{\textrm{ON INVARIANT OF THE REGULARITY INDEX OF FAT POINTS}\hfil\textrm{\thepage}}%
     \renewcommand{\@evenhead}{\@oddhead}%
     \renewcommand{\@oddfoot}{}% empty footer
     \renewcommand{\@evenfoot}{\@oddfoot}}
\makeatother     % `@' is restored as a "non-letter" character

\title{ON INVARIANT OF THE REGULARITY INDEX OF \\ FAT POINTS}         % Enter your title between curly braces
\author{Phan Van Thien \\ Department of
Mathematics,\\  University of Education, Hue University, Vietnam \\ E-mail: pvthien@hueuni.edu.vn\\}
% Enter your name between curly braces
\date{}          % Enter your date or \today between curly braces
\maketitle
% Set to use the "plain" pagestyle
\pagestyle{plain}

\begin{abstract}\noindent  We prove  invariant of the regularity index of fat points under changes of the linear subspace containing the support of the fat points (Theorem \ref{theorem35}). Then we show that Segre's bound is attained by  any set of $s$ non-degenerate equimultiple  fat points in $\mathbb P^n$, $s\le n+3$ (Theorem \ref{theorem43}). We also give an example showing that there always exists a set of $n+4$ non-degenerate equimultiple fat points in $\mathbb P^n$ such that Segre's bound is not attained (Example \ref{ex44}).
\end{abstract}



\par \section{Introduction} \ \ \ \ \ In this paper, we denote by $\mathbb P^n:=\mathbb P^n_K$ the $n$-dimensional  projective space over an algebraically closed $K$ of arbitrary characteristic, and denote by $R:=K[X_0, \ldots, X_n]$ the polynomial ring in variables $X_0, \ldots, X_n$ over $K$. If $P_1,\ldots , P_s$ are distinct points in $\mathbb P^n$; we denote by $\wp_1, \ldots, \wp_r$ the  defining homogeneous prime ideals of $P_1, \ldots, P_s$ in $R$, respectively. 

Let $P_1, \ldots, P_s$ be distinct
points in $\mathbb P^n$, and let $m_1, \ldots, m_s$ be positive integers. Then the set of all homogeneous polynomilas that vanish at $P_i$ to order $m_i$; for $i=1, \ldots, s$; is the homogeneous ideal $I:=\wp^{m_1}_1 \cap \cdots \cap
\wp^{m_s}_s$. We call the zero-scheme defined by $I$ to be fat points (or a set of  fat points) in $\mathbb P^n$, and we denote it by  $$Z:=m_1P_1+\cdots+m_sP_s.$$ 
The integer $m_i$ is called the multiplicity of $P_i$; $i=1, \ldots, s$. If $m_1=\cdots=m_s=2$, then fat points $Z=2P_1+\cdots +2P_s$ is said
to be double points. If  $m_1=\cdots=m_s=m$, then fat points $Z=mP_1+\cdots +mP_s$ is said
to be equimultiple. We call the points $P_i$ to be  the points of support of $Z$. If the support of $Z$ span $\mathbb P^n$; then $P_1, \ldots, P_s$ is called non-degenerate in $\mathbb P^n$; and we also call $Z$ a set of non-degenerate fat points in $\mathbb P^n$.

The homogeneous coordinate ring $R/I$ of $Z$ is a graded ring, $R/I=\underset{t
\ge 0}{\oplus} (R/I)_t$. 
We call the function $$H_{R/I}(t):=\dim_K (R/I)_t$$
to be the Hilbert function of $Z$ in $\mathbb P^n$. 
Since $$\dim_K I_t =\dim_K R_t - \dim_K (R/I)_t = \binom{t+n}{n}-H_{R/I}(t),$$
we know the dimension of linear space of all homogeneous polynomials of degree $t$ vanishing at $P_i$ to order $m_i$; $i=1, \ldots, s$; if we know the Hilbert function $H_{R/I}(t)$. But it is a fairly difficult problem to determine the Hilbert function of fat points in $\mathbb P^n$.

The ring $R/I$ has the multiplicity $$e(R/I)
:=\underset{i=1}{\overset{s}{\sum}} \binom{m_i+n-1}{n}.$$ The Hilbert function $H_{R/I}(t)$
strictly increases until it reaches the multiplicity $e(R/I)$, at which it stabilizes. We call 
 the least integer $t$ such that $H_{R/I}(t)=e(R/I)$ to be the regularity index of $Z$ in $\mathbb P^n$, and we denote it by $\reg(Z)$.  It is well known that
$\reg(Z)=\reg(R/I)$, the Castelnuovo-Mumford regularity of the ring $R/I$.

If $t\ge \reg(Z)$, we have
 $$\dim_K I_t = \binom{t+n}{n} - \underset{i=1}{\overset{s}{\sum}} \binom{m_i+n-1}{n}.$$
 So, it is interesting to determine $\reg(Z)$ or, even less, an upper bound for it. 
 
\par

From now on if $a$ is a rational number, we denote by $[a]$ it's integer part.

For generic fat points $Z=m_1P_1+\cdots+m_sP_s$ in $\mathbb P^2$ with $m_1 \ge
\cdots \ge m_s$, Segre \cite{Seg} showed that
$$\reg(Z) \le \max\left\{m_1+m_2-1, \left[\frac{m_1+\cdots+m_s}2\right]\right\}.$$
It was conjectured by Trung (see \cite{Th2}) and, independently, by Farabbi and Lorenzini (see \cite{FL}) that 
$$\reg(R/I) \le \max \{ T_j(Z) |\  j=1,\ldots,n \},$$
where
$$T_j(Z) := \max\left\{\left[\frac{\sum_{l=1}^q m_{i_l}+ j- 2}{j}\right] |\
P_{i_1}, \ldots , P_{i_q} \text{ on a }j\text{-dimensional linear subspace}\right\}.$$
The number $$T(Z):=\max \{ T_j(Z) |\  j=1,\ldots,n \}$$
is called the Segre's bound because it generalizes Segre's upper bound. There were many different results for proving the Segre's bound. In 2016, Ballico et al. \cite{BDP} succesfully proved the Segre's bound for $n+3$ non-degenerate  fat points in $\mathbb P^n$. Recently, Nagel and Trok (see \cite[Theorem 5.3]{NT}) have succefully proved the Segre's bound for  arbitrary fat points $Z$ in $\mathbb P^n$. 

The problem to exactly determine $\reg(R/I)$ is more fairly difficult. So far,
there are only a few results about computing $\reg(R/I)$.
\par

For abitrary fat points $Z=m_1P_1+\cdots+m_sP_s$ in $\mathbb P^n$, Davis and
Geramita \cite[Corollary 2.3]{DG} proved that
$$\reg(R/I) = m_1+\cdots+m_s-1$$  if and
only if points $P_1, \ldots, P_s$ lie on a line in $\mathbb P^n$. \par

The points $P_1, \ldots, P_s$ in $\mathbb P^n$ is said to be in linearly general position if no $j+2$ of the points $P_1, \ldots, P_s$ are on any $j$-dimensional linear subspace for $j <n$, then we also call the fat points $Z=m_1P_1+\cdots+m_sP_s$  to be in general position in $\mathbb P^n$. A rational normal curve in $\mathbb P^n$ to be a curve of degree $n$ that may be given parametrically as the image of the map
\begin{align*} \mathbb P^1 &\to \mathbb P^n.\\
(s, t) &\mapsto (s^n, s^{n-1}t, \ldots, t^n)
\end{align*}
For fat points $Z=m_1P_1+\cdots+m_sP_s$ in $\mathbb P^n$ with $m_1\ge \cdots \ge m_s$,
Catalisano et al. \cite{CTV} showed  formulas to compute $\reg(Z)$ in the
following two cases: \par If $s\ge 2$ and $P_1, \ldots, P_s$ are on a rational
normal curve in $\mathbb P^n$ \cite[Proposition 7]{CTV}, then
$$\reg(R/I) = \max\left\{m_1+m_2-1, \left[(\sum_{i=1}^s
m_i+n-2)/n\right]\right\}.$$ \par If $n\ge 3$, $2\le s \le n+2$, $2\le m_1\ge
\cdots \ge m_s$ and $P_1, \ldots, P_s$ are in general position in $\mathbb P^n$, then
$$\reg(R/I) = m_1+m_2-1.$$ \par

For fat points $Z=m_1P_1+\cdots+m_{s+2}P_{s+2}$ in $\mathbb
P^n$ with $P_1, \ldots, P_{s+2}$ not in a linear $(s-1)$-space in $\mathbb P^n$,
$s \le n$, Thien \cite[Theorem 3.4]{Th4} showed 
$$\reg(R/I) = T(Z).$$

For equimutiple fat points $Z=mP_1+\cdots+mP_s$ in $\mathbb P^n$ with $m\ne 2$; $P_1, \ldots, P_s$ are not on a linear $(r-1)$-space and $s\le r+3$, Thien and Sinh \cite[Theorem 4.5]{TS} showed
$$\reg(R/I)= T(Z).$$

Assume that $\alpha$ is a $r$-dimensional linear subspace of $\mathbb P^n$ containing the points $P_1, \ldots, P_s$. We may consider the $r$-dimensional linear subspace $\alpha$ as a projective space $\mathbb P^r$ containing the points $P_{1\alpha}:=P_1, \ldots,
P_{s\alpha}:=P_s$, and we consider $Z_{\alpha}:=m_1P_{1\alpha}+\cdots+m_sP_{s\alpha}$ as a zero-scheme in $\mathbb
P^r$. What is the relation between $\reg(Z)$ and $\reg(Z_{\alpha})$?

\par In this paper we prove that $\reg(Z)=\reg(Z_{\alpha})$ (Theorem \ref{theorem35}). Then we calculate regularity index of a set of $n+3$ non-degenerate double points in $\mathbb P^n$ (Proposition  \ref{prop42}) and show that Segre's bound is attained by  any set of $s$ non-degenerate equimultiple  fat points in $\mathbb P^n$, $s\le n+3$ (Theorem \ref{theorem43}). We also show that there always exists a set of $n+4$ non-degenerate equimultiple fat points in $\mathbb P^n$ such that Segre's bound is not attained (Corollary \ref{cor45}).

\section{Preliminaries} \ \ \ \

We use the following lemmas which have been
proved yet. The first lemma allows us to
compute the regularity index by induction.

\begin{lemma}\label{lem21} \cite[Lemma 1]{CTV} Let $P_1,\ldots , P_r, P$ be
distinct points in $\mathbb P^n$ and $\wp$ be the defining ideals in $R$ corresponding to the point $P$. If
$m_1,\ldots , m_r, a$ are positive integers, $J := \wp^{m_1}_1\cap
\cdots\cap  \wp^{m_r}_r$, and $I = J \cap \wp^a$, then $$\reg(R/I) =
\max\left\{a-1, \reg(R/J), \reg(R/(J+\wp^a)) \right\}.$$ \end{lemma}

Note that $R/(J+\wp^a)$ is a zero-dimensional graded ring algebra, the regularity index $\reg(R/(J+\wp^a))$ is the least integer such that $[R/(J+\wp^a))]_t=0$.
To estimate  $\reg(R/(J+\wp^a))$ we shall  use the following lemma. \par
\begin{lemma}\label{lem22} \cite[Lemma 3]{CTV}
 Let $P_1,\ldots , P_r, P$ be distinct
points in $\mathbb P^n$ and $\wp$ be the defining ideal in $R$ corresponding to the point $P$. If $m_1,\ldots , m_r, a$ are positive integers, $J =
\wp^{m_1}_1\cap \cdots\cap  \wp^{m_r}_r$ and $\wp = (X_1,\ldots , X_n)$, then
$$\reg(R/(J+\wp^a)) \le b$$ if and only if $X^{b-i}_0M \in J+\wp^{i+1}$ for
every monomial $M$ of degree $i$ in $X_1,\ldots , X_n$; $i = 0,\ldots , a-1$.
\end{lemma}  \par

\begin{lemma}\label{lem23} \cite[Corollary 2]{CTV}
 Let $s\ge 2$; $P_1, \ldots, P_s$ be distinct points in $\mathbb P^n$ and $m_1 \ge \cdots \ge m_s$ be positive integers. If $I=\wp_1^{m_1}\cap \cdots \cap  \wp_s^{m_s}$, then 
 $$\reg(R/I)\ge m_1+m_2-1.$$
\end{lemma}  \par

\begin{lemma}\label{lem24} \cite[Corollary 8]{CTV}
Let $n\ge 3$, $2\le s \le n+2$, and let $P_1, \ldots, P_s$ be distinct points in general position in $\mathbb P^n$. If $2\le m_1\ge
\cdots \ge m_s>0$ are given integers and $I=\wp_1^{m_1} \cap \cdots \cap \wp_s^{m_s}$, then
$$\reg(R/I) = m_1+m_2-1.$$ \par

\end{lemma}

\par\smallskip

In \cite{BDP}, Ballico et al. succesfully proved the Segre's  for non-degenerate $n+3$ fat points in $\mathbb P^n$.\par

\begin{lemma}\label{lem25} \cite[Theorem 2.1]{BDP} Let $Z=m_1P_1+\cdots+m_{n+3}P_{n+3}$ be a scheme of fat points supported on a non-degenerate set of distinct points in $\mathbb P^n$. Then $Z$ sastifies the Segre's bound, namely
$$\reg(Z) \le \max \{ T(Z, L): L \subseteq \mathbb P^n \},$$
where $L$ is a $j$-dimensional linear subspace of $\mathbb P^n$; $j=1, \ldots, n$; and
$$T(Z, L) := \max\left\{\left[\frac{\sum_{l=1}^q m_{i_l}+ j- 2}{j}\right] |\
P_{i_1}, \ldots , P_{i_q} \in L \right\}.$$
\end{lemma}\par

We see that the number $\max \{ T(Z, L): L \subseteq \mathbb P^n \}$ is the Segre's bound $T(Z)$ in the Trung's conjecture.

The following lemma was proved in \cite{Th4}, it also has been shown in \cite[Lemma 3.1]{NT}. \par

\begin{lemma}\label{lem26} \cite[Lemma 3.3]{Th4} Let $X=\{P_1,\ldots , P_s\}$ be
a set of distinct points in $\mathbb P^n$ and $m_1,\ldots , m_s$ be positive
integers. Put $I = \wp^{m_1}_1\cap \cdots\cap \wp^{m_s}_s$. If $Y$ is a subset of $X$ and $J =\underset{P_i \in Y}{\cap} \wp_i^{m_i}$, then $$\reg(R/J) \le \reg(R/I).$$
\end{lemma}\par

This implies that if $Z=m_1P_1+\cdots +m_sP_s$ is the set of fat points defining by $I$, and $U=\underset{P_i\in Y}{\sum}m_iP_i$ is the set of fat points defining by $J$, then 
$$\reg(U) \le \reg(Z).$$

\begin{lemma}\label{lem27} \cite[Theorem 3.4]{Th4} Let $P_1, \ldots, P_{s+2}$ be
distinct points not on a $(s-1)$-dimensional linear subspace in $\mathbb P^n$; $s \le n$; and 
$m_1, \ldots, m_{s+2}$ be positive integers. Put $I=\wp_1^{m_1} \cap \cdots \cap\wp_{s+2}^{m_{s+2}}$. Then, $$\reg(R/I) = \max \{ T_j |\  j=1,\ldots,n \},$$
where
$$T_j:= \max\left\{\left[\frac{\sum_{l=1}^q m_{i_l}+ j- 2}{j}\right] |\
P_{i_1}, \ldots , P_{i_q} \text{ on a }j\text{-dimensional linear subspace}\right\}.$$
\end{lemma}\par

If $P_1, \ldots, P_s$ are on a $r$-dimensional linear subspace $\alpha \subset \mathbb P^n$ and there are no $j+2$ of the points $P_1, \ldots, P_s$ on any $j$-dimensional linear subspace for $j<r$; then we call $P_1, \ldots, P_s$ to be in linearly general position on $\alpha$. The following lemmas were proved in \cite{TS}.

\begin{lemma}\label{lem28} \cite[Theorem 3.1]{TS} Let $P_1, \ldots, P_{s}$ be
distinct points in general position on an $r$-dimensional linearly linear subspace $\alpha$ in $\mathbb P^n$, $s \le r+3$. Let $m_1, \ldots, m_s$ be positive integers and $Z=m_1P_1+\cdots+m_sP_s$.  Then, $$\reg(Z) = \max\{T_1, T_r\},$$ 
where \begin{align*}
T_1&=\max\{m_i+m_j-1\ |\ i\ne j; i, j=1, \ldots, s\},\\
T_r&=\left[\frac{m_1+\cdots +m_{s} +r-2}{r} \right]. \end{align*}
\end{lemma}\par

We see that the Segre's bound is attained for this case. If $r=n$ and $s=n+3$, the $P_1, \ldots, P_{n+3}$ are in linearly general position in $\mathbb P^n$. So, there exists a rational normal curve passing through $P_1, \ldots, P_{n+3}$. Then we get Proposition 7 in \cite{CTV}.

\begin{lemma}\label{lem29} \cite[Corollary 3.2]{TS} Let $P_1, \ldots, P_s$ be distinct points in $\mathbb P^n$, $s\le 5$. Let $m$ be a positive integer and $Z=mP_1+\cdots+mP_s$. Then 
$$\reg(Z) = \max \{ T_j |\  j=1,\ldots,n \},$$
where
$$T_j:= \max\left\{\left[\frac{qm+ j- 2}{j}\right] |\
P_{i_1}, \ldots , P_{i_q} \text{ on a }j\text{-dimensional linear subspace}\right\}.$$ 
\end{lemma}\par

\begin{lemma}\label{lem210} \cite[Theorem 4.5]{TS} Let $X=\{P_1, \ldots, P_{s+3}\}$ be a set of 
distinct points not on a $(s-1)$-dimensional linear subspace in $\mathbb P^n$, $s \le n$ and $m$ be a positive integer, $m\ne 2$. Let $Z=mP_1+\cdots+mP_{s+3}$ be a set equimultiple fat points.  Then,
$$\reg(Z) = \max \{ T_j |\ j=1,\ldots, n \},$$ where
$$T_j = \max\left\{\left[\frac{mq+ j- 2}{j}\right] |\ P_{i_1},
\ldots , P_{i_q} \text{ lie on a linear }j\text{-space}\right\},$$
\end{lemma}\par

We see that in case $s=n$ and $m\ne 2$, the regularity index of $n+3$ non-degenerate equimultiple fat points $Z=m1P_1+\cdots+mP_{n+3}$ in $\mathbb P^n$ reaches the Segre's bound.

\section{Invariant of the regularity index under changes of linear subspace containing the support of fat points}

\ \ \ \  Let $Z=m_1P_1+ \cdots +m_sP_s$ be a set of fat points in $\mathbb P^n$, and $I$ be the ideal consists all homogeneous polynomials in $R$ vanishing at $P_i\in \mathbb P^n$ to order $m_i$; $i=1, \ldots, s$. The regularity index $\reg(Z)$ is the least integer such that the Hilbert function
$$H_{R/I}(t)=\dim_K (R/I)_t=\binom{t+n}{n}-\dim_K I_t$$
reaches the multiplicity $e(R/I)=\underset{i=1}{\overset{s}{\sum}} \binom{m_i+n-1}{n}.$
Note that the Hilbert function of an ideal is invariant under projective automorphisms. So, the regularity index of a set of fat points is also invariant under projective automorphisms.

Assume that $\alpha$ is a $r$-dimensional linear subspace of $\mathbb P^n$ containing the points $P_1, \ldots, P_s$. We consider the $r$-dimensional linear subspace $\alpha$ as a projective space $\mathbb P^r$ containing the points $P_{1\alpha}:=P_1, \ldots,
P_{s\alpha}:=P_s$, and consider $Z_{\alpha}:=m_1P_{1\alpha}+\cdots+m_sP_{s\alpha}$ as a set of fat points  in $\mathbb
P^r$.  We denote by $I_{\alpha}$ the ideal consists all homogeneous polynomials in $R_{\alpha}:=K[X_0, \ldots, X_r]$ that vanish at $P_{i\alpha}$ to order $m_i$; $i=1, \ldots, s$. The Hilbert function  of $Z_{\alpha}$ in $\mathbb P^r$ is 
$$H_{Z_{\alpha}}(t):=H_{R_{\alpha}/I_{\alpha}}(t)=\dim_K (K[X_0, \ldots, X_r]/I_{\alpha})_t=\binom{t+r}{r}-\dim_K (I_{\alpha})_t.$$
The multiplicity number of $Z_{\alpha}$ in $\mathbb P^r$ is $$e(R_{\alpha}/I_{\alpha}):=\underset{i=1}{\overset{s}{\sum}} \binom{m_i+r-1}{r}.$$
The regularity index of $Z_{\alpha}$ in $\mathbb P^r$ is the least integer $t$ such that
$H_{R/I_{\alpha}}(t)=e(R_{\alpha}/I_{\alpha})$, and we denote it by $\reg(Z_{\alpha})$.

\bigskip
We have $e(R_{\alpha}/I_{\alpha}) \le e(R/I)$, ${\displaystyle \binom{t+r}{r} \le \binom{t+n}{n}}$ and $\dim_K (I_{\alpha})_t \le  \dim_K I_t$. 
Here there are two questions: 
\bigskip

1) Question 1: $H_{Z_{\alpha}}(t)=H_Z(t)$?

2) Question 2: $\reg(Z_{\alpha})=\reg(Z)$?

\bigskip

For Question 1, we now do not know about the relation between $H_Z(t)$ and $H_{Z_{\alpha}}(t)$ if $t < \reg(Z)$. But if $t\ge \reg(Z)$, then $H_{Z_{\alpha}}(t) \le H_Z(t)$. Furthermore, if there exists $m_i\ge 2$ and $r<n$, then $H_{Z_{\alpha}}(t) < H_Z(t)$ for $t\ge \reg(Z)$. (see Corollary \ref{cor36a}).

For Question 2, due to a result of Benedetti et al. (see \cite[Lemma 4.4]{BFL}), we know that if $Z=m_1P_1+\cdots+m_sP_s$ is a set of fat points in $\mathbb P^n$ whose support is contained in a $(n-1)$-dimensional linear subspace $\alpha \cong \mathbb P^{n-1}$ and if there is a non-negative integer $t$ such that $\reg(Z_{\alpha}) \le t,$ then $$\reg(Z) \le t.$$ By inductive argument on the dimension of linear subspace, we get

\begin{lemma}\label{lem31} Let $Z=m_1P_1+\cdots+m_sP_s$ be a set of fat points in
$\mathbb P^n$ whose support is contained in a linear $r$-space $\alpha \cong
\mathbb P^r$. We may consider the linear $r$-space $\alpha$ as a $r$-dimensional
projective space $\mathbb P^r$ containing the points $P_{1\alpha}:=P_1, \ldots,
P_{s\alpha}:=P_s$ and consider $Z_{\alpha}=m_1P_{1\alpha}+\cdots+m_sP_{s\alpha}$ as a set of fat points in $\mathbb P^r$. If there is a non-negative integer $t$ such that $$\reg(Z_{\alpha}) \le t,$$
then
$$\reg(Z) \le t.$$ \end{lemma}\par

Now we prove that the converse is also true. The fisrt, we prove the following lemma.

\begin{lemma}\label{lem32}  Let $\alpha$ be a $(n-1)$-dimensional linear subspace in $\mathbb P^n$, and let $P_1, \ldots, P_s$ be distinct points on $\alpha$. Let $m_1, \ldots, m_s$ be positive integers. We may consider the $(n-1)$-dimensional linear subspace $\alpha$ as a $(n-1)$-dimensional
projective space $\mathbb P^r$ containing the points $P_{1\alpha}:=P_1, \ldots,
P_{s\alpha}:=P_s$, and consider $Z_{\alpha}=m_1P_{1\alpha}+\cdots+m_sP_{s\alpha}$ as a set of fat points in $\mathbb P^{n-1}$. Put $R_{\alpha}=K[X_0, \ldots, X_{n-1}]$. Let $\wp_j \subset R$ be the defining prime ideal of $P_j\in \mathbb P^n$, and $\wp_{j\alpha} \subset R_{\alpha}$ be the defining prime ideal of $P_{j\alpha}\in \mathbb P^r$; $j=1, \ldots, s$. Put $J=\wp_1^{m_1} \cap \cdots \cap \wp_{s-1}^{m_{s-1}}$ and $J_{\alpha}=\wp_{1\alpha}^{m_1} \cap \cdots \cap \wp_{(s-1)\alpha}^{m_{s-1}}$. If there is a non-negative integer $b$ such that 
$$\reg(R_{\alpha}/(J_{\alpha}+\wp_{s\alpha}^{m_s})) > b, $$ then
$$\reg(R/(J+\wp_s^{m_s})) > b.$$
\end{lemma}
\begin{proof} Consider distinct points $P_1, \ldots, P_s$ on $\alpha$, an $(n-1)$-dimensional linear subspace in $\mathbb P^n$. We may assume that  $n$ points $(\underset{1}{\underbrace{1}}, 0, \ldots, 0, 0, 0), \ldots, (0, 0, \ldots, 0,\underset{n}{\underbrace{1}}, 0)$ span $\alpha$. Put $P_s=(1, 0, \ldots, 0)$, then $\wp_s=(X_1, \ldots, X_n)$, $P_{s\alpha}=(1, 0, \ldots, 0)\in \mathbb P^{n-1}$ and $\wp_{s\alpha}=(X_1, \ldots, X_{n-1})$. If 
$$\reg(R_{\alpha}/(J_{\alpha}+\wp_{s\alpha}^{m_s})) > b, $$
then by Lemma \ref{lem22} there exists a monomial $M=X_1^{c_1} \cdots X_{n-1}^{c_{n-1}} \in \wp_{s\alpha}$, $c_1+\cdots+c_r=i$, $i\in \{0, \ldots, m_s-1\}$, such that $$X_0^{b-i}M \notin J_{\alpha}+\wp_{s\alpha}^{i+1}.$$
Now we consider $M=X_1^{c_1} \cdots X_{n-1}^{c_{n-1}}$ in $K[X_0, \ldots, X_n]$. We claim that $$X_0^{b-i}M \notin J+\wp_s^{i+1}.$$ 

In fact, if $X_0^{b-i}M \in J+\wp_s^{i+1}$, then $X_0^{b-i}M \in \left[ J+\wp_s^{i+1}\right]_b$. Consider the finite dimensional vector space $[J+\wp_s^{i+1}]_b$ over the field $K$ with the monomial $X_0^{b-i}M$ being an element of it's basis, we get $$X_0^{b-i}M \in J \text{ or }X_0^{b-i}M \in \wp_s^{i+1}.$$
In case of $X_0^{b-i}M \in J$, then $X_0^{b-i}M=X_0^{b-i}X_1^{c_1} \cdots X_{n-1}^{c_{n-1}} \in K[X_0, \ldots, X_{n-1}]$ and $X_0^{b-i}M$ vanishing at the points $P_j\in \alpha$ to order $m_j$; $j=1, \ldots, s$. Thus $X_0^{b-i}M \in J_{\alpha}$, this is a contradiction with $X_0^{b-i}M \notin J_{\alpha}+\wp_{s\alpha}^{i+1}$. In case of $X_0^{b-i}M \in \wp_s^{i+1}$, then $X_0^{b-i}M=X_0^{b-i}X_1^{c_1} \cdots X_{n-1}^{c_{n-1}} \in K[X_0, \ldots, X_{n-1}]$ and $X_0^{b-i}M$ vanishing at the points $P_s\in \alpha$ to order $i+1$. Thus $X_0^{b-i}M \in \wp_{s\alpha}^{i+1}$, this is a contradiction with $X_0^{b-i}M \notin J_{\alpha}+\wp_{s\alpha}^{i+1}$. Thus, we have the claim
$$X_0^{b-i}M \notin J+\wp_s^{i+1}.$$ Then, by Lemma \ref{lem22} we get
$$\reg(R/(J+\wp_{s\alpha}^{m_s})) > b.$$
\end{proof}

From the above lemma,  by inductive argument on the dimension of linear subspace, we get the following lemma.

\begin{lemma}\label{lem33}  Let $\alpha$ be an $r$-dimensional linear subspace in $\mathbb P^n$, $1 \le r \le n-1$, and let $P_1, \ldots, P_s$ be distinct points on $\alpha$. Let $m_1, \ldots, m_s$ be positive integers. We may consider the $r$-dimensional linear subspace $\alpha$ as a $r$-dimensional
projective space $\mathbb P^r$ containing the points $P_{1\alpha}:=P_1, \ldots,
P_{s\alpha}:=P_s$, and $Z_{\alpha}=m_1P_{1\alpha}+\cdots+m_sP_{s\alpha}$ as a zero-scheme in $\mathbb
P^r$. Put $R_{\alpha}=K[X_0, \ldots, X_r]$. Let $\wp_j \subset R$ be the defining prime ideal of $P_j$, and $\wp_{j\alpha} \subset R_{\alpha}$ be the defining prime ideal of $P_{j\alpha}$; $j=1, \ldots, s$. Put $J=\wp_1^{m_1} \cap \cdots \cap \wp_{s-1}^{m_{s-1}}$ and $J_{\alpha}=\wp_{1\alpha}^{m_1} \cap \cdots \cap \wp_{(s-1)\alpha}^{m_{s-1}}$. If there is a non-negative integer $b$ such that 
$$\reg(R_{\alpha}/(J_{\alpha}+\wp_s^{m_s})) > b, $$ then
$$\reg(R/(J+\wp_{s\alpha}^{m_s})) > b.$$
\end{lemma}

\bigskip

By using the above lemmas we prove the following result.

\begin{proposition}\label{prop34}  Let $Z=m_1P_1+\cdots+m_sP_s$ be fat points in
$\mathbb P^n$ whose support is contained in an $r$-dimensional linear space $\alpha \cong
\mathbb P^r$. We may consider the $r$-dimensional linear space $\alpha$ as a projective space $\mathbb P^r$ containing the points $P_{1\alpha}:=P_1, \ldots,
P_{s\alpha}:=P_s$, and $Z_{\alpha}=m_1P_{1\alpha}+\cdots+m_sP_{s\alpha}$ as a zero-scheme in $\mathbb
P^r$. If there is a non-negative integer $b$ such that 
$$\reg(Z_{\alpha}) > b, $$ then
$$\reg(Z) > b.$$
\end{proposition}

\begin{proof} Without loss of generality we may assume that $m_1 \ge \ldots \ge m_s$. Put $R_{\alpha}=K[X_0, \ldots, X_r]$ and denote by $\wp_{j\alpha}$ the homogeneous prime ideal in $R_{\alpha}$ defining by $P_{j\alpha}$; $j=1, \ldots, s$. We argue by induction on $s$. If $s=1$, then $Z=m_1P_1$ in $\mathbb P^n$ and $Z_{\alpha}=m_1P_{1\alpha}$ in $\mathbb P^r$. It is well known that $\reg(m_1P_1)=m_1-1$ and $\reg(m_1P_{1\alpha})=m_1-1$. Suppose that the lemma is true for $s-1$. Put $J=\wp_1^{m_1} \cap \cdots \cap \wp_{s-1}^{m_{s-1}}$, $I=J \cap \wp_s^{m_s}$,  $J_{\alpha}=\wp_{1\alpha}^{m_1} \cap \cdots \cap \wp_{(s-1)\alpha}^{m_{s-1}}$, $I_{\alpha}=J_{\alpha} \cap \wp_{s\alpha}^{m_s}$. By Lemma \ref{lem21} we get 
 $$\reg(R/I) =
\max\left\{m_s-1, \reg(R/J), \reg(R/(J+\wp_s^{m_s})) \right\}$$
and
$$\reg(R_{\alpha}/I_{\alpha}) =
\max\left\{m_s-1, \reg(R_{\alpha}/J_{\alpha}), \reg(R_{\alpha}/(J_{\alpha}+\wp_{s\alpha}^{m_s})) \right\}.$$

By Lemma \ref{lem23} we have $\reg(R/J) \ge m_1+m_2-1$ and $\reg(R_{\alpha}/J_{\alpha}) \ge m_1+m_2-1$. So, $\reg(R/J) \ge m_s-1$ and $\reg(R_{\alpha}/J_{\alpha})\ge m_s-1$. Therefore, 
$$\reg(R/I) =
\max\left\{\reg(R/J), \reg(R/(J+\wp_s^{m_s})) \right\}$$
and
$$\reg(R_{\alpha}/I_{\alpha}) =
\max\left\{\reg(R_{\alpha}/J_{\alpha}), \reg(R_{\alpha}/(J_{\alpha}+\wp_{s\alpha}^{m_s})) \right\}.$$

If $\reg(R_{\alpha}/I_{\alpha}) > b$ in $\mathbb P^r$, then we consider two following cases:

\bigskip

{\it Case $\reg(R_{\alpha}/J_{\alpha}) > b$:} Then by inductive assumption we have
$\reg(R/J) > b.$
Therefore, $$\reg(R/I) >  b.$$

\bigskip

{\it Case $\reg(R_{\alpha}/(J_{\alpha}+\wp_{s\alpha}^{m_s})) > b$:} Then by using Lemma \ref{lem33} we get 
$\reg(R/(J+\wp_s^{m_s})) > b$. Therefore, $$\reg(R/I) > b.$$
\end{proof}

From Lemma \ref{lem31} and Proposition \ref{prop34} we get the following theorem which answers Question 2.

\begin{theorem}\label{theorem35} Let $Z=m_1P_1+\cdots+m_sP_s$ be fat points in
$\mathbb P^n$ whose support is contained in a linear $r$-space $\alpha \cong
\mathbb P^r$. We may consider the linear $r$-space $\alpha$ as a $r$-dimensional
projective space $\mathbb P^r$ containing the points $P_{1\alpha}:=P_1, \ldots,
P_{s\alpha}:=P_s$, and consider $Z_{\alpha}=m_1P_{1\alpha}+\cdots+m_sP_{s\alpha}$ as a zero-scheme in $\mathbb
P^r$. Denote $\reg(Z)$ the regularity index of $Z$ in $\mathbb P^n$, and denote $\reg(Z_{\alpha})$ the regularity of $Z_{\alpha}$ in $\mathbb P^n$. Then
$$\reg(Z_{\alpha})=\reg(Z).$$

\end{theorem}
Hence, the regularity index of a set of fat points is invariant under changes of linear subspace containing the support of fat points.

\bigskip

\begin{corollary}\label{cor36a} Let $Z=m_1P_1+\cdots+m_sP_s$ be fat points in
$\mathbb P^n$ whose support is contained in a linear $r$-space $\alpha \cong
\mathbb P^r$. We may consider the linear $r$-space $\alpha$ as a $r$-dimensional
projective space $\mathbb P^r$ containing the points $P_{1\alpha}:=P_1, \ldots,
P_{s\alpha}:=P_s$, and consider $Z_{\alpha}=m_1P_{1\alpha}+\cdots+m_sP_{s\alpha}$ as a zero-scheme in $\mathbb
P^r$.  If $t\ge \reg(Z)$, then 
$$H_{Z_{\alpha}}(t) \le H_Z(t).$$
Furthermore, if there exists $m_i\ge 2$ and $r<n$, then $$H_{Z_{\alpha}}(t) < H_Z(t)$$ for $t\ge \reg(Z)$.

\end{corollary}
\begin{proof} We know that the Hilbert function $H_Z(t)$ strictly increases until it reaches the multiplicity $e(Z)$, at which it stabilizes, and  the regularity index $\reg(Z)$  is  the least integer $t$ such that $H_{R/I}(t)=e(Z)$. So, if $t\ge \reg(Z)$ then $H_Z(t)=e(Z)=\underset{i=1}{\overset{s}{\sum}} \binom{m_i+n-1}{n}$. 

For fat points $Z_{\alpha}$ in $\mathbb P^r$, we also get: if $t\ge \reg(Z_{\alpha})$, then $H_{Z_{\alpha}}(t)=\underset{i=1}{\overset{s}{\sum}} \binom{m_i+r-1}{r}$.  By Theorem \ref{theorem35}, 
$\reg(Z)=\reg(Z_{\alpha})$. Therefore, if $t\ge \reg(Z)=\reg(Z_{\alpha})$ then
$$H_{Z_{\alpha}}(t)=\underset{i=1}{\overset{s}{\sum}} \binom{m_i+r-1}{r} \le
 \underset{i=1}{\overset{s}{\sum}} \binom{m_i+n-1}{n}=H_Z(t).$$
This result implies that if $m_1=\cdots=m_s=1$ then $H_{Z_{\alpha}}(t) = H_Z(t)$ for $t\ge \reg(Z)$, and if there exists $m_i\ge 2$ and $r<n$, then $H_{Z_{\alpha}}(t) < H_Z(t)$ for $t\ge \reg(Z)$.

\end{proof}

\bigskip

\begin{corollary}\label{cor36}
Let $Z=m_1P_1+\cdots+m_sP_s$ be fat points in $\mathbb P^n$ with the support in a $r$-dimensional linear subspace $\alpha$. If $t\ge \reg(Z)$, then we have the relation between dimension of the $t$-th graded part $I_t$ of ideal $I$ consisting all homogeneous polynomials in $K[X_0, \ldots, X_n]$ that vanish at $Pi$ to order $m_i$; $i=1, \ldots, s$; and dimension of the $t$-th graded part $(I_{\alpha})_t$ of ideal $I_{\alpha}$ consisting all homogeneous polynomials in $K[X_0, \ldots, X_r]$ that vanish at $Pi$ to order $m_i$; $i=1, \ldots, s$; as following:
\begin{align*}
&(t+r+1)\cdots (t+n)\left[\underset{i=1}{\overset{s}{\sum}} \binom{m_i+r-1}{r} +\dim_K(I_{\alpha})_t \right]\\&=(r+1)\cdots (n-1)n\left[\underset{i=1}{\overset{s}{\sum}} \binom{m_i+n-1}{n} + \dim_K I_t \right].
\end{align*}
\end{corollary}
\begin{proof} We may consider $\alpha$ as a $r$-dimensional
projective space $\mathbb P^r$ containing the points $P_{1\alpha}:=P_1, \ldots,
P_{s\alpha}:=P_s$, and consider $Z_{\alpha}=m_1P_{1\alpha}+\cdots+m_sP_{s\alpha}$ as a set of fat points in $\mathbb P^r$. Since $I=\underset{t \ge 0}{\oplus} I_t$ is the ideal consists all homogeneous polynomials in $K[X_0, \ldots, X_n]$ that vanish at $Pi$ to order $m_i$, $i=1, \ldots, s$; the Hilbert function of $Z$ is
$$H_Z(t)=\dim_K (K[X_0, \ldots, X_n]/I)_t=\dim_K K[X_0, \ldots, X_n]_t-\dim_K I_t=\binom{t+n}{n}-\dim_K I_t.$$
Since $I_{\alpha}=\underset{t
\ge 0}{\oplus} (I_{\alpha})_t$ is the ideal consists all homogeneous polynomials in $K[X_0, \ldots, X_r]$ that vanish at $P_{i\alpha}$ to order $m_i$, $i=1, \ldots, s$;  the Hilbert function of $Z_{\alpha}$ is
$$H_{Z_{\alpha}}(t)=\dim_K (K[X_0, \ldots, X_r]/I_{\alpha})_t=\dim_K K[X_0, \ldots, X_r]_t-\dim_K (I_{\alpha})_t=\binom{t+r}{r}-\dim_K (I_{\alpha})_t.$$
By Theorem \ref{theorem35} we have $\reg(Z)=\reg(Z_{\alpha})$. If $t\ge \reg(Z)=\reg(Z_{\alpha})$, then $H_Z(t)$ reaches the multiplicity $e(Z)=\underset{i=1}{\overset{s}{\sum}} \binom{m_i+n-1}{n}$, and $H_{Z_{\alpha}}(t)$ reaches the multiplicity $e(Z_{\alpha})=\underset{i=1}{\overset{s}{\sum}} \binom{m_i+r-1}{r}$, and so we have

$${\displaystyle \begin{cases}
\binom{t+n}{n}=\underset{i=1}{\overset{s}{\sum}} \binom{m_i+n-1}{n}+\dim_K I_t,\\
\binom{t+r}{r}=\underset{i=1}{\overset{s}{\sum}} \binom{m_i+r-1}{r}+\dim_K (I_{\alpha})_t.
\end{cases}
}$$
Hence if $t\ge \reg(Z)$, then
\begin{align*}
&(t+r+1)\cdots (t+n)\left[\underset{i=1}{\overset{s}{\sum}} \binom{m_i+r-1}{r} +\dim_K(I_{\alpha})_t \right]=(t+r+1)\cdots (t+n) \binom{t+r}{r}\\
&=(r+1)\cdots (n-1)n \binom{t+n}{n}=(r+1)\cdots (n-1)n\left[\underset{i=1}{\overset{s}{\sum}} \binom{m_i+n-1}{n} + \dim_K I_t \right].
\end{align*}
\end{proof}

\bigskip

Let $Z=m_1P_1+\cdots+m_sP_s$ be a set of fat points in $\mathbb P^n$, and let $\alpha$ be a $r$-dimensional linear subspace. We put $$\alpha \cap Z:=\underset{P_i \in \alpha}{\sum} m_iP_i$$ and
$$w_{\alpha \cap Z}:=\underset{P_i \in \alpha}{\sum} m_i.$$

\begin{corollary}\label{cor37}
Let $Z$ be a set of fat points in $\mathbb P^n$. Assume that $\alpha$ is a $r$-dimensional linear subspace such that ${\displaystyle \left[ \frac{w_{\alpha\cap Z} -2}{r} \right]=T(Z)}$. Consider $(\alpha \cap Z)_{\alpha}$ as a set of fat points in $\mathbb P^r$. If the support of $(\alpha \cap Z)_{\alpha}$ lies on a line or on a rational normal curve of $\mathbb P^r$, then $\reg(Z)=T(Z)$. 
\end{corollary}

\begin{proof}

By Theorem \ref{theorem35} we have $$\reg(\alpha \cap Z)=\reg((\alpha \cap Z)_{\alpha}).$$
By using Lemma \ref{lem26} and \cite[Theorem 5.3]{NT} we get
$$\reg(\alpha \cap Z) \le \reg(Z) \le T(Z).$$
Since the support of $(\alpha \cap Z)_{\alpha}$ lies on a line or on a rational normal curve of $\mathbb P^r$, by using \cite[Corollary 2.3]{DG} and \cite[Proposition 7]{CTV} we get $$\reg((\alpha \cap Z)_{\alpha}) \ge \left[ \frac{w_{\alpha\cap Z} -2}{r}\right].$$
Note that ${\displaystyle \left[ \frac{w_{\alpha\cap Z} -2}{r} \right]=T(Z)}$ by the hypothesis. 
Therefore, $$\reg(Z)=\reg((\alpha \cap Z)_{\alpha})=T(Z).$$
 
\end{proof}

\bigskip

Let $Z=m_1P_1+\cdots+m_sP_s$ be a set of fat points in $\mathbb P^n$. If $\alpha$ is a linear subspace of $\mathbb P^n$ such that $P_1, \ldots, P_s$ span $\alpha$, then we say that $Z$ is non-degenerate in $\alpha$. From Theorem \ref{theorem35}, we have the following remark.

\begin{remark}\label{rem38} To calculate the regularity index of a set of fat points $Z$ in $\mathbb P^n$, we find a linear subspace $\alpha \cong \mathbb P^r$ such that $Z$ is non-degenerate on $\alpha$. Then we calculate the regularity index of the non-degenerate fat points $Z_{\alpha}$ in $\mathbb P^r$.
\end{remark}

\bigskip

\section{Segre's bound and the regularity index of non-degenerate equimultiple fat points}

Recall that for a set of fat points $Z=m_1P_1+\cdots+m_sP_s$ in $\mathbb P^n$, the Segre's bound for the regularity index of $Z$ is the number

$$T(Z)= \max \{ T_j(Z) |\  j=1,\ldots,n \},$$
where
$$T_j(Z) := \max\left\{\left[\frac{\sum_{l=1}^q m_{i_l}+ j- 2}{j}\right] |\
P_{i_1}, \ldots , P_{i_q} \text{ on a }j\text{-dimensional linear subspace}\right\}.$$

\begin{lemma}\label{lem41} Let $Z=2P_1+\cdots+2P_7$ be a set of seven non-degenerate double points in $\mathbb P^4$. Then
$$\reg(Z)\ge 4.$$
\end{lemma}
\begin{proof} Let $U=2Q_1+\cdots+2Q_7$ be a set of seven generic double points in $\mathbb P^4$. Denote by $H_Z(t)$ the Hilbert funtion of $Z$ and by $H_U(t)$ the Hilbert function of $U$. By the maximality of the Hilbert function of generic fat points we have
$$H_Z(t) \le H_U(t).$$
Note that $Z$ and $U$ have the same multiplicity
$$e(Z)=e(U)= 7\binom{2+4-1}{4}=35.$$
By the Alexander-Hirschowits  Theorem \cite{AH} about the Hilbert funtion of generic double points in $\mathbb P^4$ we have
$$H_U(3)=34.$$
Since $H_Z(3)\le H_U(3)=34 < 35=e(Z)$, by the definition of $\reg(Z)$ we get
$$\reg(Z)\ge 4.$$
\end{proof} 

The above lemma and Theorem \ref{theorem35} help us to calculate the regularity index of a set of $n+3$ non-degenerate double points in $\mathbb P^n$. In this case we see that the regularity index reaches the Segre's bound.

\begin{proposition}\label{prop42} Let $Z=2P_1+\cdots+2P_{n+3}$ be a set of $n+3$ non-degenerate double points in $\mathbb P^n$. Then 
$$\reg(Z)= T(Z),$$
where
$$T(Z)= \max \{ T_j(Z) |\  j=1,\ldots,n \},$$
$$T_j(Z) := \max\left\{\left[\frac{2q+ j- 2}{j}\right] |\
P_{i_1}, \ldots , P_{i_q} \text{ on a }j\text{-dimensional linear subspace}\right\}.$$

\end{proposition} 
\begin{proof} Consider the fat points in the assumption of Lemma \ref{lem25}, we see that
\begin{align*} T(Z, L) &= \max\left\{\left[\frac{\sum_{l=1}^q m_{i_l}+ j- 2}{j}\right] |\
P_{i_1}, \ldots , P_{i_q}\in L \text{ is a $j$-dimensional linear subspace}\right\}\\
&=T_j(Z).\end{align*} 
Hence $$\max \{ T(Z, L): L \subseteq \mathbb P^n \}=T(Z).$$
By using  Lemma \ref{lem25} we get $\reg(Z)\le T(Z)$.  So, it is sufficient to prove $$\reg(Z) \ge T(Z).$$
We argue by induction on $n$. For $n=2$, then $n+3 = 5$. By Lemma \ref{lem29} we get $\reg(Z)=T(Z)$. For $n\ge 3$, we assume that the proposition is true for any $m+3$ non-degenerate double points $2P_1+\cdots+2P_{m+3}$ in $\mathbb P^m$, $m<s$. We now consider $n+3$ non-degenerate double points $Z=2P_1+\cdots+2P_{n+3}$ in $\mathbb P^n$.

Since  $P_1, \ldots, P_{n+3}$ are not not on any $(n-1)$-dimensional linear subspace, there are at most $j+3$ points of $\{P_1, \ldots, P_{n+3}\}$ on a $j$-dimensional linear subspace of $\mathbb P^n$; $j=1, \ldots, n-1$. Put
$$p=\min\{j\ | \ T_j(Z)=T(Z)\}.$$
and denote by $\alpha$  the $p$-dimensional linear subspace containing points $P_{i_1}, \ldots, P_{i_r}$ of $\{P_1, \ldots, P_{n+3}\}$ such that 
$$\left[ \frac{2r+p-2}p \right]=T_p(Z)=T(Z).$$
Then the points  $P_{i_1}, \ldots, P_{i_r}$ are not on any $(p-1)$-dimensional linear subspace. So, $r\ge p+1$. Since  $P_1, \ldots, P_{n+3}$ are non-degenerate in $\mathbb P^n$, there are at most $p+3$ points of $\{P_1, \ldots, P_{n+3}\}$ on a $p$-dimensional linear subspace. Hence $r\le p+3$.

Consider the set of $r$ double points 
$$Y=2P_{i_1}+\cdots+2P_{i_{r}}.$$ 
in $\mathbb P^n$. Put $P_{i_j\alpha}=P_j$, $j=1, \ldots, r$. Consider the set of of $r$ non-degenerate double points 
$$Y_{\alpha}=2P_{i_1\alpha}+\cdots+2P_{i_r\alpha}$$
in $\mathbb P^p$. If $r=p+1$ or $r=p+2$, then by Lemma \ref{lem27} we  get
$$\reg(Y_{\alpha})=T(Y_{\alpha})=\left[ \frac{2r+p-2}p \right].$$
So, $\reg(Y_{\alpha})=T(Z).$ By Theorem \ref{theorem35} we get $\reg(Y)=\reg(Y_{\alpha})$. By use Lemma \ref{lem26}, we obtain $$\reg(Z) \ge \reg(Y)=T(Z).$$
If $r=p+3$, then we consider two following cases for $p$:

\medskip

\noindent {\it Case $p<n$:} Since $p+3$ points $P_{i_1}, \ldots, P_{i_{p+3}}$ lie on the $p$-dimensional linear subspace $\alpha$ and do not lie on a $(p-1)$-dimensional linear subspace, by the inductive assumption we have $$\reg(Y_{\alpha})=T(Y_{\alpha}).$$
By using Theorem \ref{theorem35} we get $\reg(Y)=\reg(Y_{\alpha})$.
Now use Lemma \ref{lem26} and obtain $\reg(Z) \ge \reg(Y)$. Since $\reg(Y)=\reg(Y_{\alpha})=T(Y_{\alpha})=\left[ \frac{2r+p-2}p \right]=T(Z)$, we get
$$\reg(Z)\ge T_p(Z)=T(Z).$$

\medskip

\noindent {\it Case $p=n$:} Then $T_n(Z)>T_j(Z)$; $j=1, \ldots, n-1$. By defining we have
$$T_n(Z)=\left[\frac{(n+3)2+n-2}n\right]=3+\left[\frac{4}n \right], \text{ and } T_1(Z)\ge 2.2-1=3.$$
Since $T_n(Z) > T_1(Z)$, we have $n\le 4$. But we are considering $n\ge 3$, so $n=3$ or $n=4$.

For $n=3$, then $T_3(Z)=4=T(Z)$. Consider the set of $6$ non-degenerate double points $$Z=2P_1+\cdots+2P_6$$ in $\mathbb P^3$. If $P_1, \ldots, P_{6}$ are not in general position in $\mathbb P^3$, then there are $3$ points of $\{P_1, \ldots, P_{6}\}$ on a line or there are $4$ points of $\{P_1, \ldots, P_{6}\}$ on a $2$-dimensional linear subspace. So, $T_1(Z)\ge 5$ or $T_2(Z) \ge 4$, this is a contradiction with $T(Z)=T_p(Z)=T_3(Z)=4 > T_j(Z); j=1, 2$. So, $P_1, \ldots, P_{6}$ are in general position in $\mathbb P^3$. By using the Lemma \ref{lem28} we get $$\reg(Z)=T(Z).$$

For $n=4$, then $T_4(Z)=4=T(Z)$. Consider the set of seven non-degenerate double points $$Z=2P_1+\cdots+2P_7$$ in $\mathbb P^4.$ By Lemma \ref{lem41} we get $$\reg(Z) \ge 4=T(Z).$$

The proof of Proposition \ref{prop42} is completed.
\end{proof}

The Proposition \ref{prop42} help us to prove the following theorem which shows that the Segre's bound is attained for the regularity index of any $s+3$ non-degenerate equimultiple fat points in $\mathbb P^n$, $s\le n+3$.

\begin{theorem}\label{theorem43} Let $m$ be a positive integer and $Z=mP_1+\cdots+mP_{s}$ be a set of $s$ non-degenerate equimultiple points in $\mathbb P^n$. If $s\le n+3$, then  
$$\reg(Z)= T(Z),$$
where
$$T(Z)= \max \{ T_j(Z) |\  j=1,\ldots,n \},$$
$$T_j(Z) := \max\left\{\left[\frac{mq+ j- 2}{j}\right] |\
P_{i_1}, \ldots , P_{i_q} \text{ on a }j\text{-dimensional linear subspace}\right\}.$$
\end{theorem}
\begin{proof} Since $P_1, \ldots, P_s$ are not on any $(n-1)$-dimensional linear subspace, we have $s\ge n+1$. We consider the three following cases for $s$.

\medskip

\noindent {\it Case $s=n+1$:} Then $P_1, \ldots, P_{n+1}$ are in general position in $\mathbb P^n$. By Lemma \ref{lem24} we get $$\reg(Z)=2m-1=T_1(Z)=T(Z).$$

\medskip

\noindent {\it Case $s=n+2$:} Then $P_1, \ldots, P_{n+2}$ are not on a $(n-1)$-dimensional linear subspace in $\mathbb P^n$. By Lemma \ref{lem27} we get $$\reg(Z)=T(Z).$$

\medskip
\noindent {\it Case $s=n+3$:} If $m\ne 2$, then by using Lemma \ref{lem210} we get
$$\reg(mP_1+\cdots+mP_{n+3})=T(Z).$$
If $m=2$, then by Proposition \ref{prop42} we get
$$\reg(2P_1+\cdots+2P_{n+3})= T(Z).$$

\end{proof}

The following simple example shows that Theorem \ref{theorem35} is not true for $n=2$, $s=n+4$.

From now on, if $H$ is a hyperplane in $\mathbb P^n$, we coincide $H$ with the linear form in $R$ defining it. So, if $H_j$ is a hyperplane passing through points $P_j$ and $m_j$ is a positive integer; $j=1, \ldots, r$; then we get $H_1^{m_1}\cdots H_r^{m_r}\in \wp_1^{m_1} \cap \cdots \cap \wp_r^{m_r}$.

\begin{example}\label{ex44} \textnormal{ Let $l_1$, $l_2$ be two distinct lines in $\mathbb P^2$ and 
$P_1, \ldots, P_6$ be six distinct points such that $P_1, P_2, P_3\in l_1\setminus l_2$; $P_4, P_5\in l_2\setminus l_1$ and $P_6 \notin l_1 \cup l_2$}. 

\bigskip

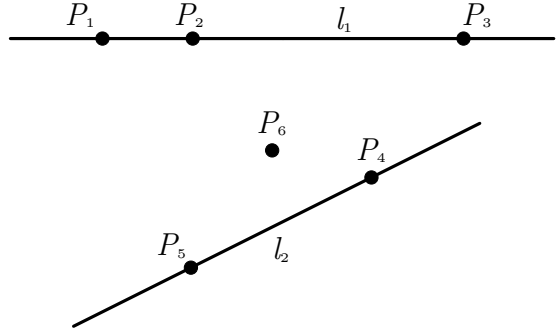
\begin{center}
\begin{tikzpicture}[line cap=round,line join=round,>=triangle 45,x=1.2cm,y=1.2cm]
\clip(6.97321718931476, 2.543530778164918) rectangle (13.6,6.6);

\draw [line width=1.2pt] (7.3,5.94)-- (8.32,5.94);
\draw [line width=1.2pt] (8.32,5.94)-- (10.32,5.94);
\draw [line width=1.2pt] (10.32,5.94)-- (12.32,5.94);
\draw [line width=1.2pt] (12.32,5.94)-- (13.32,5.94);

\draw [line width=1.2pt] (8.0,2.75)-- (9.3,3.4);
\draw [line width=1.2pt] (9.3,3.4)-- (11.3,4.4);
\draw [line width=1.2pt] (11.3,4.4)-- (12.5,5.0);

%\draw [line width=1.2pt] (10.32,6.94)-- (9.3,3.4);
%\draw [line width=1.2pt] (10.32,6.94)-- (11.3,3.4);
\begin{scriptsize}
\draw [fill=black] (8.32,5.94) circle (2.5pt);
\draw[color=black] (8.075423925667838,6.2) node {$\text{\large P}_{1}$};
\draw [fill=black] (9.32,5.94) circle (2.5pt);
\draw[color=black] (9.283089430894323,6.2) node {$\text{\large P}_2$};
\draw [fill=black] (12.32,5.94) circle (2.5pt);
\draw[color=black] (12.47031358885019,6.2) node {$\text{\large P}_3$};
\draw [fill=black] (11.3,4.4) circle (2.5pt);
\draw[color=black] (11.3,4.7) node {$\text{\large P}_4$};
\draw [fill=black] (9.3,3.4) circle (2.5pt);
\draw[color=black] (9.0871080139384,3.644494773519156) node {$\text{\large P}_5$};
\draw [fill=black] (10.2,4.7) circle (2.5pt);
\draw[color=black] (10.2,5.0) node {$\text{\large P}_6$};

\draw[color=black] (11,6.15) node {$\text{\large l}_1$};
\draw[color=black] (10.3,3.6) node {$\text{\large l}_2$};

\end{scriptsize}
\end{tikzpicture}\\
\end{center}

\bigskip

\textnormal{ Let $m$ be a positive integer, then $Z=mP_1+\cdots+mP_6$ is a set of $6$ non-degenerate equimultiple fat points in $\mathbb P^2$. We have
$$T_1(Z)=3m-1, \ T_2(Z)=\left[ \frac{6m}{2}\right]=3m=T(Z).$$}

\textnormal{Let $\wp_j$ be the homogeneous prime ideal in $\mathbb R=K[X_0, X_1, X_2]$ corresponding to the point $P_j$; $j=1, \ldots, 6$. Choose $P_6=(1, 0, 0)$, then $\wp_6=(X_1, X_2)$. Put $J=\wp_1^m\cap \cdots \cap \wp_5^m$, $I=J\cap \wp_6^m$. Then $R/I$ is the coordinate ring of $Z$, so $\reg(Z)=\reg(R/I)$, the Castelnuovo-Mumford regularity index of the ring $R/I$. By Lemma \ref{lem21} we get
 $$\reg(R/I) =
\max\left\{m-1, \reg(R/J), \reg(R/(J+\wp_6^m)) \right\}.$$}

\textnormal{ Since $l_1$ is a line passing through three points $P_1, P_2, P_3$; we have $l_1^m\in \wp_1^m\cap \wp_2^m\cap \wp_3^m$. Since $l_2$ is a line passing through two points $P_4, P_5$; we have $l_2^m\in \wp_4^m\cap \wp_5^m$. Hence, $l_1^m l_2^m \in J$. This implies that for every monomial $M$ of degree $i$ in $X_1, X_2$; $i=0, \ldots, m-1$; we have $$l_1^ml_2^m M \in J.$$
Since $l_j$ is a line avoiding $P_6=(1, 0, 0)$, we can write $l_j=X_0+g_j$ for some $g_j\in \wp_6=(X_1, X_2)$; $j=1, 2$. Thus we have
$$(X_0+g_1)^m (X_0+g_2)^m M \in J.$$
Moreover, since $g_1, g_2\in \wp_6$ and $M\in \wp_6^i$, we get $X_0^{2m}M \in J+\wp_6^{i+1}$
for $i=0, \ldots, m-1$. This implies that $$X_0^{3m-1-i}M \in J+\wp_6^{i+1}$$ for $i=0, \ldots, m-1$. By Lemma \ref{lem22} we get $$\reg(R/(J+\wp_6^2))\le 3m-1.$$}

\textnormal{Let $U=mP_1+\cdots+mP_5$, 
we have $$T_1(U)=3m-1 \ge T_2(U)=\left[ \frac{5m}2\right]=2m+\left[ \frac{m}2\right].$$
So $T(U)=T_1(U)=3m-1$. By Lemma \ref{lem29} we get $\reg(U)=3m-1.$ Since $R/J$ is the homogeneous coordinate ring of $U$. Thus $$\reg(R/J)=\reg(U)=3m-1.$$ 
Therefore,  $$\reg(Z)=\reg(R/I)=\max\{m-1, 3m-1, \reg(R/(J+\wp_6^m))\}=3m-1 < T(Z).$$}
\end{example}

By using Theorem \ref{theorem35} and Example \ref{ex44} we show that Theorem \ref{theorem43} is not true for $s=n+4$. The Segre's bound is not attained for this case.

\begin{corollary}\label{cor45} Let $m, n$ be positive integers, $n\ge 2$. There always exists a set of $n+4$ non-degenerate equimultiple fat points  $Z=mP_1+\cdots+mP_{n+4}$ in $\mathbb P^n$ such that  
$$\reg(Z)=3m-1 < T(Z)=3m,$$
where
$$T(Z)= \max \{ T_j(Z) |\  j=1,\ldots,n \},$$
$$T_j(Z) := \max\left\{\left[\frac{mq+ j- 2}{j}\right] |\
P_{i_1}, \ldots , P_{i_q} \text{ on a }j\text{-dimensional linear subspace}\right\}.$$
\end{corollary}
\begin{proof} We consider the set of fat points $Z$ with the configuration as following.

Let $\alpha_1$ be the $2$-dimensional linear subspace in $\mathbb P^n$ defining by $n-2$ linear forms $X_3, \ldots, X_n$; and $l_1$, $l_2$ be two distinct lines in $\alpha_1 \cong \mathbb P^2$. Let $P_1, \ldots, P_3$ be  three distinct points on $l_1\setminus \l_2$; and $P_4, P_5$ be two distinct points on $l_2\setminus l_1$. Let $P_6$ be a points in $\alpha_1$ and not on $l_1\cup l_2$. Consider the set of $6$ double points $Z_{\alpha_1}=mP_1+\cdots+mP_6$ in $\mathbb P^2$, by Example \ref{ex44} we get
$$\reg(Z_{\alpha_1})=3m-1 \text{ and } T(Z_{\alpha_1})=T_{2}(Z_{\alpha_1})=3m.$$

Since $P_4, P_5, P_6$ do not lie on a line, we can choose $P_7, \ldots, P_{n+4}$ be distinct points in $\mathbb P^n$ such that $n+1$ points $P_4, P_5, P_6, P_7, \ldots, P_{n+4}$ span $\mathbb P^n$. Then $P_4, \ldots, P_{n+4}$ are in linearly general position in $\mathbb P^n$ and $$Z=mP_1+\cdots+mP_{n+4}$$ is the set of $n+4$ non-degenerate equimultiple fat points in $\mathbb P^n$. We have
\begin{align*}
T_1(Z)&=3m-1, T_2(Z)=3m, \\
T_j(Z)&=\left[ \frac{6m+(j-2)m+j-1}j\right]=m+\left[ \frac{4m+j-2}j\right]; j=3, \ldots, n-1.
\end{align*}
Thus, $$T(Z)=T_2(Z)=3m.$$ 
For $j=2, \ldots, n-1$; the  points  $P_4, P_5, P_6, \ldots, P_{j+5}$ span a $(j+1)$-dimensional linear subspace, say $\alpha_j$.  Put $P_{i\alpha_j}=P_i$; $i=1, \ldots, j+5$; then
$$Z_{\alpha_j}=mP_{1\alpha_j}+ \cdots+mP_{(j+5)\alpha_j}$$
is the set of $j+5$ non-degenerate equimultiple fat points in $\mathbb P^j$. Since $\alpha_{n-1}=\mathbb P^n$, we get $$\reg(Z)=\reg(Z_{\alpha_{n-1}}).$$

We argue by induction on $j$ to prove that $$\reg(Z_{\alpha_j})=\reg(Z_{\alpha_1})=3m-1$$ for $j=2, \ldots, n-1$.

For $j=2$, consider the set of $7$ non-degenerate equimultiple fat points 
$$Z_{\alpha_2}=m_1P_{1\alpha_2}+\cdots+mP_{7\alpha_2}$$
in $\mathbb P^3$. Put $J_{\alpha_2}=\wp_{1\alpha_2}^m \cap \cdots \cap \wp_{6\alpha_2}^m$ and $I_{\alpha_2}=J_{\alpha_2} \cap \wp_{7\alpha_2}^m$, $R_{\alpha_2}=K[X_0, X_1, X_2, X_3]$. Then $\reg(Z_{\alpha_2})= \reg(R_{\alpha_2}/I_{\alpha_2})$. By Lemma \ref{lem21} we get
$$\reg(R_{\alpha_2}/I_{\alpha_2})=\max\{m-1, \reg(R_{\alpha_2}/J_{\alpha_2}), \reg(R_{\alpha_2}/(J_{\alpha_2}+\wp_{7\alpha_2}^m))\}.$$
Since $P_{1\alpha_2}=P_1, \ldots, P_{6\alpha_2}=P_6$ are on $\alpha_1 \cong \mathbb P^2$; by using Example \ref{ex44} we get
$$\reg(R_{\alpha_2}/J_{\alpha_2})=\reg(Z_{\alpha_1})=3m-1.$$
Let $P_{7\alpha_2}=(1, 0, 0, 0)$ in $\mathbb P^3$. Since $P_{7\alpha_2} \notin \alpha_1$ and $P_{1\alpha_2}, \ldots, P_{6\alpha_2} \in \alpha_1$, we can choose a $2$-dimensional linear subspace $H$ in $\mathbb P^3$ passing through $P_{1\alpha_2}, \ldots, P_{6\alpha_2}$ and avoiding $P_{7\alpha_2}$. Then we get $$H^m \in \wp_{1\alpha_2}^m \cap \cdots \cap \wp_{6\alpha_2}^m=J_{\alpha_2}.$$ It follows that
$$H^{2m-i}M \in J_{\alpha_2}$$
for every monomial $M$ of degree $i$ in $X_1, X_2, x_3$; $i=0, \ldots, m-1$. Since $P_{7\alpha_2}\notin H$, we can write $H=X_0+G$ for some linear form $G\in \wp_{7\alpha_2}=(X_1, X_2, X_3)$. Thus
we have $$(X_0+G)^{2m-i} M \in J_{\alpha_2}$$
for every monomial $M$ of degree $i$ in $X_1, X_2, x_3$; $i=0, \ldots, m-1$. Moreover, since $G\in \wp_{7\alpha_2}$ and $M\in \wp_{7\alpha_2}^i$, we get
$$X_0^{2m-i} M \in J_{\alpha_2}+\wp_{7\alpha_2}^{i+1}$$
for every monomial $M$ of degree $i$ in $X_1, X_2, x_3$; $i=0, \ldots, m-1$. By Lemma \ref{lem22} we get
$$\reg(R_{\alpha_2}/(J_{\alpha_2}+\wp_{7\alpha_2}^m))\le 2m.$$
Therefore, $$\reg(Z_{\alpha_2})=3m-1.$$
Assume that $\reg(Z_{\alpha_{n-2}})=3m-1$. We need prove that $\reg(Z_{\alpha_{n-1}})=3m-1$. Consider the set of $n+4$ non-degenerate equimultiple fat points $$Z=mP_1+\cdots+mP_{n+4}$$ in $\alpha_{n-1}=\mathbb P^n$. Put $J=\wp_1^m \cap \cdots \cap \wp_{n_3}^m$ and $I=J\cap \wp_{n+4}^m$. Then $\reg(Z)=\reg(R/I)$. By Lemma \ref{lem21} we get
$$\reg(R/I)=\max\{m-1, \reg(R/J), \reg(R/(J+\wp_{n+4}^m))\}.$$
By Theorem \ref{theorem35} and inductive assumption we get $\reg(R/J)=3m-1$. Choose $P_{n+4}=(1, 0, \ldots, 0)$. Since $\alpha_{n-2}$ is a $(n-1)$-dimensional linear subspace passing through $P_1, \ldots, P_{n+3}$, we have
$$\alpha_{n-2}^mN \in J$$
for every monomial $N$ of degree $i$ in $X_1, \ldots, X_n$; $i=0, \ldots, m-1$.
By a similar argue as above, since $P_{n+4}\notin \alpha_2$, we can write $\alpha_{n-2}=X_0+L$ for some linear form $L\in \wp_{n+4}$. Thus, we have $$(X_0+L)^{2m-i}N \in J.$$ Moreover, since $L\in \wp_{n+4}$ and $N\in \wp_{n+4}^i$, we get
$$X_0^{2m-i}N \in J+\wp_{n+4}^{i+1}$$ for every monomial $N$ of degree $i$ in $X_1, \ldots, X_n$; $i=0, \ldots, m-1$. By Lemma \ref{lem22} we get
$$\reg(R/(J+\wp_{n+4}^{m})) \le 2m-1.$$
Therefore, $$\reg(Z)=3m-1.$$

\end{proof}

% Set the ending of a LaTeX document

\end{document}